\documentclass[a4paper,12pt]{amsart}
\usepackage[margin=1in]{geometry}
\usepackage{amsfonts,hyperref}
\usepackage{amsmath}
\usepackage{amssymb}
\usepackage{amsthm}
\usepackage{epigraph}
\usepackage[T1]{fontenc}
\usepackage[dvips]{graphicx}
\usepackage[latin1]{inputenc}
\usepackage{color}
\usepackage[all]{xypic}
\usepackage{fancyhdr}
\usepackage{natbib}
\usepackage{mathrsfs}
\usepackage{etoolbox} 

\newcommand{\beas}{{\begin{eqnarray*}}}
\newcommand{\eeas}{{\end{eqnarray*}}}

\theoremstyle{definition}
\newtheorem{theorem}[equation]{Theorem}      
\newtheorem{lemma}[equation]{Lemma}          %
\newtheorem{corollary}[equation]{Corollary}  

\newtheorem{conj}[equation]{Conjecture}

\theoremstyle{definition}

\numberwithin{equation}{section}

\let\into=\hookrightarrow

\let\tensor=\otimes
\newcommand{\C}{{\mathbb{C}}}
\newcommand{\sD}{\mathscr{D}}
\let\sD=\sD
\newcommand{\End}{{\rm End}}

\newcommand{\Q}{{\mathscr{Q}}}
\newcommand{\SL}{{\rm SL}}

\newcommand{\gr}{{\rm gr}}
\newcommand{\hit}{{\rm Hitchin}}
\newcommand{\id}{{\rm id}}

\newcommand{\mydot}{{\small\bullet}}
\newcommand{\oper}{{\rm Oper}}
\newcommand{\rk}{{\rm rk}}

\newcommand{\sDgz}{{\sD}_{{g,0}}}
\newcommand{\sDrgz}{{\sD}^r_{{g,0}}}
\newcommand{\tDrgz}{\widetilde{\sD}^r_{g,0}}

\newcommand{\sMgnb}{{\bar{\mathscr{M}}_{g,n}}}
\newcommand{\sMgn}{{\mathscr{M}_{g,n}}}

\newcommand{\sMgz}{\sM_{g,0}}
\newcommand{\sM}{{\mathscr{M}}}
\newcommand{\sNgnb}{\bar{\mathscr{N}}_{g,n}}

\newcommand{\sNgz}{\sN_{g,0}}
\newcommand{\sN}{\mathscr{N}}
\newcommand{\sSgnb}{{{\bar{\sS}}}_{g,n}}

\newcommand{\sS}{{\mathscr{S}}}
\newcommand{\spec}{{\rm Spec}}
\renewcommand{\O}{{\mathscr{O}}}
\renewcommand{\P}{{\mathbb{P}}}
\newcommand{\be}{\begin{equation}}
\newcommand{\ee}{\end{equation}}
\newcommand{\sV}{{\mathscr{V}}}
\newcommand{\sU}{{\mathscr{SU}}}

\let\isom=\simeq
\newcommand{\Deg}{{\rm deg}}
\newcommand{\pgl}{\rm PGL}

\makeatletter

\pretocmd{\NAT@citex}{%
  \let\NAT@hyper@\NAT@hyper@citex
  \def\NAT@postnote{#2}%
  \setcounter{NAT@total@cites}{0}%
  \setcounter{NAT@count@cites}{0}%
  \forcsvlist{\stepcounter{NAT@total@cites}\@gobble}{#3}}{}{}
\newcounter{NAT@total@cites}
\newcounter{NAT@count@cites}
\def\NAT@postnote{}

\def\NAT@hyper@citex#1{%
  \stepcounter{NAT@count@cites}%
  \hyper@natlinkstart{\@citeb\@extra@b@citeb}#1%
  \ifnumequal{\value{NAT@count@cites}}{\value{NAT@total@cites}}
    {\ifNAT@swa\else\if*\NAT@postnote*\else%
     \NAT@cmt\NAT@postnote\global\def\NAT@postnote{}\fi\fi}{}%
  \ifNAT@swa\else\if\relax\NAT@date\relax
  \else\NAT@@close\global\let\NAT@nm\@empty\fi\fi
  \hyper@natlinkend}
\renewcommand\hyper@natlinkbreak[2]{#1}

\patchcmd{\NAT@citex}
  {\ifNAT@swa\else\if*#2*\else\NAT@cmt#2\fi
   \if\relax\NAT@date\relax\else\NAT@@close\fi\fi}{}{}{}
\patchcmd{\NAT@citex}
  {\if\relax\NAT@date\relax\NAT@def@citea\else\NAT@def@citea@close\fi}
  {\if\relax\NAT@date\relax\NAT@def@citea\else\NAT@def@citea@space\fi}{}{}

\makeatother

\begin{document}
\title{The degree of the dormant operatic locus}
\author{Kirti Joshi}
\address{Math. department, University of Arizona, 617 N Santa Rita, Tucson
85721-0089, USA.} \email{kirti@math.arizona.edu}


\begin{abstract}
Let $X$ be a smooth, projective curve of genus $g\geq 2$ over an algebraically closed field of characteristic $p>0$. I provide a conjectural formula for the degree of the scheme of dormant ${\rm PGL}(r)$-opers on $X$ where $r\geq 2$  (I assume that $p$ is greater than an explicit constant depending on $g,r$). For $r=2$ a dormant ${\rm PGL}(2)$-oper is a dormant indigenous bundle on $X$ in the sense of Shinichi Mochuzki (and his work provides a formula only for $g=2,r=2,p\geq 5$, from a different point of view). In 2014, Yasuhiro Wakabayashi has shown that my conjectural formula holds for $r=2,g\geq 2$ and $p>2g-2$ and more recently he has proved the conjecture in all ranks for  generic curves of genus  at least two.
\end{abstract}

\maketitle 
\epigraph{Hara-naka ya\\
Mono nimo tsukazu\\
Naku hibari}{Matsuo Basho \citep*{miura}}

\section{Introduction}
Let $k$ be  an algebraically closed field of characteristic $p>0$.
In his fundamental papers \citep*{mochizuki96,mochizuki-foundations} Shinichi Mochizuki introduced and studied several (log) stacks over the moduli stack, $\sMgnb$, of stable log-curves of type $(g,n)$ (by \citep*{fkato00} this log stack is isomorphic to the stack of stable, n-pointed curves of genus $g$ equipped with the log struture obtained from the divisor at infinity). Amongst them is the stack of $\sSgnb\to \sMgnb$ of \emph{indigenized curves of genus $g$ and $n$ marked points} (another name for $\sSgnb$ is the stack of Schwarz-torsors). An \emph{indigenized curve} is a stable log-curve of type $(g,n)$ equipped with an indigenous bundle. A \emph{nil-curve of type $(g,n)$} is a stable log curve of type $(g,n)$ which is equipped with an indigenous bundle whose underlying connection is nilpotent with nilpotent residues at marked points. The stack $\sNgnb\to\sMgnb$ of \emph{nil-curves} is a closed substack of the stack of indigenized curves.  It was shown in \citep*{mochizuki96} that $\sNgnb\to\sMgnb$ is finite flat of degree $p^{3g-3}$. This is a fundamental result in the subsequent development of Mochizuki's $p$-adic Teichmuller Theory (see \citep*{mochizuki-foundations}).

One may consider, as I do in the rest of this paper, the restriction of these stacks to $\sMgn$ the moduli stack of smooth, proper curves with $n$-marked points and from now on assume $g\geq 2$ and $n=0$. I do this, primarily, because the theory of opers over stable log-curves is not yet available, though it is clear to me that the construction of these stacks given in \citep*{mochizuki96}, for $r=2$, extends \emph{mutatis mutandis}  to the case $r\geq 2$.

Before proceeding further, I recall a few facts about opers (see \citep*{bd-hitchin}, \citep*{beilinson05}; opers in characteristic $p>0$ are also studied in \citep*{joshi06}, \citep*{joshi09}). Let me note that following ideas of \citep*{mochizuki96} one may relativize the \citep*{bd-hitchin} construction of the stack of opers on a smooth, proper curve of genus $g$ to construct the stack of \emph{operized} curves over $\sMgz$ (i.e. a curve equipped with an oper on it) or more generally the stack of \emph{operized stable log-curves of type $(g,n)$}. But I do not pursue this line of thought here.

 Let $X/k$ be a smooth, projective curve of genus $g\geq2$ on $X$. Let $r\geq 2$ be an integer.  In this situation  one has at one's disposal the stack of $\pgl(r)$-opers, denoted $\oper_r(X)$, on $X$ (this stack is in fact a scheme). For $r=2$ the stack of $\pgl(2)$-opers  is the stack of indigenous bundles studied by Mochizuki in \citep*{mochizuki96}. Thus opers provide a natural generalization of indigenous bundles.  In \citep*{joshi09} some of Mochizuki's results were extended to all ranks, and in particular we defined, what we call the \emph{Hitchin-Mochizuki morphism}:
\begin{eqnarray*}
\oper_r(X)&\to& \hit_r(X)=\oplus_{i=2}^rH^i(X,(\Omega_X^{1})^{\tensor i})\\
(V,\nabla,V_\mydot)&\mapsto&\psi(V,\nabla)^{1/p}
\end{eqnarray*}
which maps an oper to the $p$-th root of the characteristic polynomial of the $p$-curvature of the oper connection. Let $\sNgz^r(X)$ be the schematic fibre over zero of the Hitchin-Mochizuki morphism. Let me call $\sNgz^r(X)$ the \emph{scheme of nilpotent $\pgl(r)$-opers  on $X$}.  A result of \citep*{joshi09}  is that $\sNgz^r(X)$ is finite (and in  we gave a conjectural degree of this map to be $p^{\dim \pgl(r)(g-1)}$). For $r=2$ this is \citep*{mochizuki96}.

The principal object of interest for this note is the stack $\sDgz^r\to\sMgz$ of \emph{dormant $\pgl(r)$-opers on $X$} which is a closed substack of $\sNgz^r$, consisting of nilpotent opers on $X$ whose underlying indigenous bundle is \emph{dormant} i.e., the $p$-curvature of the connection is zero. This is the \emph{dormant operatic locus} of the title.  The closed subscheme of $\sDgz^r(X)\subset \sNgz^r(X)$ defined by the equation $\psi=0$ (i.e. by the vanishing of the $p$-curvature of the oper connection). For brevity I will shorten  \emph{dormant $\pgl(r)$-opers on $X$} to \emph{dormant opers on $X$}.  The dormant locus $\sDgz^r(X)\subset \sNgz^r(X)$ is also a finite subscheme of $\oper_r(X)$. For $r=2$ this was proved by Mochizuki. Note that my notation for the dormant locus is different from that used in \citep*{mochizuki-foundations} as it is easier to remember: $\sN$ for nilpotent and $\sD$ for dormant.

 One is interested in the degree of $\sDgz^r(X)$. In the present note I describe a conjectural formula for this degree.  For $r=2,g=2,p\geq 5$ this was proved by \citep*{mochizuki-foundations},\citep*{osserman07} and \citep*{lange08} by completely different methods.

The provenance of this formula is as follows. In a ``back-of-the-envelop'' calculation done for my NSF grant proposal for 2006 (the proposal was not funded), I had observed that the calculations by \citep*{lange08}, \citep*{mochizuki-foundations},
\citep*{osserman07}, for the degree of the dormant (tautologically operatic)
locus of rank two bundles for an \textit{ordinary} (or \textit{general}) curve of
genus two and calculation of Gromov-Witten invariants of certain
Quot-schemes due to \citep*{lange03}, \citep*{holla04} agreed. In this setup
\citep*{lange08} had calculated the degree of the dormant locus (as a
degree of the Quot-scheme which I have described explicitly in this
note) as a certain Chern class; the fact that these two numbers
coincide was unequivocally established (in this case). This observed coincidence for $g=2,r=2$ was the basis for  Conjecture~\ref{dormant-degree-formula}. Let me note that this conjecture was based on one observed data point. At that time the left hand side of the formula made sense only for $r=2$ (by Mochizuki's work).
Subsequent to my work with Christian Pauly (\citep*{joshi09}) the
left hand side of the conjectural formula now makes sense for all
$p>C(r,g)$ (for a certain explicit constant $C(r,g)$ which I will explain in
the text). 

Recently Yasuhiro Wakabayashi  has given an ingenious proof of Conjecture~\ref{dormant-degree-formula} for  $X$ general, $r=2,g\geq 2$ and $p>2g-2$ in \citep*{wakabayashi13} (part of his PhD Thesis at RIMS). For more comments on Wakabayashi's work see Section~\ref{addcomments}. After this paper was submitted, Wakabayashi has also recently announced a proof of the generic case of Conjecture~\ref{dormant-degree-formula} for all ranks and genus (see \citep{wakabayashi14}). Wakabayashi's \textit{remarkable and foundational work} also establishes the theory of opers over pointed stable curves and also makes substantial progress towards the speculation voiced in this paper about ``fusion calculus'' for the degrees of dormant loci in different rank and genera, analogous to fusion calculus for the classical Verlinde formula. In a series of papers (available at \url{www.arxiv.org}) written subsequent to \textit{loc. cit.}, Wakabayashi has proved a number of new and original results which provide a deeper insight into this fascinating subject.

For several years, since my first observation in 2006, I had  lost interest in pursuing or publishing the set of ideas presented here. I am deeply indebted to Shinichi Mochizuki for suggesting that my conjectural formula be published, and also for many conversations around this subject and for his invitation to spend some time at the Research Institute for Mathematical Sciences (RIMS), Kyoto.  I would also like to thank Akio Tamagawa for numerous conversations around many topics of common
interest. I am also grateful to RIMS for providing hospitality
and an excellent working environment. 

I also thank the anonymous referee for a careful and diligent reading of this paper, pointing out a number of  deficiencies of the earlier version of this paper, correcting many typos and providing the reference \citep{marian07}. The referee's effort has led to an enormous improvement in the readability of this paper. Thanks are also due to Brian Conrad for his support and his patience (as an editor) in dealing with the long delays (on my part) in completing revision(s) of this manuscript.

\section{Notations}
Let $k$ be an algebraically closed field of characteristic $p>0$. Let $S/k$ be a $k$-scheme.
Let $X\to S$ be a proper, smooth morphism of relative dimension one over $S$.   Let
$\sigma:S\to S$ be the absolute Frobenius morphism of $S$. Let
$X^{(1)}=X\times_{k,\sigma}S$. Then one has a ``standard''
Frobenius diagram (which commutes!) with $F=F_{X/S}$ the relative Frobenius morphism of $X\to S$:
\begin{equation}
\xymatrix{
  X \ar@/_/[ddr]_{f} \ar@/^/[drr]^{F_{abs}}
    \ar[dr]^{F}                   \\
   & X^{(1)} \ar[d]^{f^{(1)}} \ar[r]_{\sigma}
                      & X \ar[d]_{f}    \\
   & S \ar[r]^{\sigma}     & S
   }
\end{equation}

Now suppose that $S=\spec(k)$. Let $J_X$ be the
Jacobian of $X$, let $J_X[n]$ be the closed subscheme of $J_X$ which
is the kernel of multiplication by $n$.
Let $\sV:J_X\to J_X$ be the verschiebung.


\section{Opers}
From now on fix a smooth, proper curve $X/k$ of genus $g\geq 2$. Let $\Omega_X^1$ be the sheaf of regular one forms on $X$.

Assume that $p>C(r,g)=r(r-1)(r-2)(g-1)$ (I do this because I will use a crucial result of \citep*{joshi09} which requires this assumption about $p$, though most of the general facts about opers which I recall below can be proved without this assumption. Opers were introduced in \citep*{bd-hitchin}. Reader will find \citep*{beilinson05} quite useful. Opers in characteristic $p$ were introduced in \citep*{joshi06} and studied in greater detail in \citep*{joshi09}. I refer the reader to these references for basic facts on opers.

An $\SL(r)$ oper on $X$ is a triple $(V,V_\mydot,\nabla)$ with $(\det(V),\det(\nabla))$ isomorphic to $(\O_X,\nabla=d)$ and a decreasing flag $V_\mydot$ of subbundles which satisfies Griffiths transversality with respect to $\nabla$ and for all $1\leq i\leq r$, the $\O_X$-linear morphism induced by $\nabla$:
\begin{equation}
\gr^i(V)\to \gr^{i-1}(V)\tensor \Omega^1_X
\end{equation}
is an isomorphism of line bundles.

A $\pgl(r)$-oper is the projectivization $\P(V)$ of an $\SL(r)$-oper $(V,\nabla,V_\mydot)$ and  together with the induced connection on $\P(V)$ and the induced reduction of structure group to a Borel subgroup.

The stack, $\oper_r(X)$, of $\pgl(r)$-opers on $X$ is in fact a scheme, isomorphic to the Hitchin space $\oplus_{i=2}^r H^0(X,(\Omega_X^1)^{\tensor i})$ (see \citep*{bd-hitchin,beilinson05,joshi09}).

The Hitchin-Mochizuki morphism of \citep*{joshi09} is the morphism

\begin{equation}
\oper_r(X)\to \hit_r(X)=\oplus_{i=2}^rH^i(X,(\Omega_X^{1})^{\tensor i})
\end{equation}
which maps an oper to the $p$-th root of the $p$-curvature.
Let $\sNgz^r\subset\oper_r(X)$ be the schematic fibre over zero of the Hitchin-mochizuki morphism. Then $\sNgz^r$  is finite (as a scheme) (see \citep*{joshi09}) and in fact the Hitchin-Mochizuki morphism is finite. Let $\sDrgz\subset\sNgz^r$ be the closed subscheme of dormant opers, i.e., the subscheme defined by the vanishing of the $p$-curvature.

\section{A Quot scheme} Let us fix, for once and for all, a line bundle $L$ on $X$
such that
\begin{equation}\label{l-condition}
L^{\tensor r}\tensor(\Omega^1_X)^{\tensor \frac{r(r-1)}{2}}=\O_X.
\end{equation} Let us
note that $\frac{r(r-1)}{2}$ is always an integer and as $k$ is
algebraically closed such a line bundle $L$ always exists. Clearly
such an $L$ is unique up to tensoring by a line bundle of degree
$r$. Thus the scheme of such line bundles is evidently a torsor over
$J_X[r]$.

Let $\Q^{r,0}_L$ be the quotient scheme $Quot^{r,0}_L$ defined by
its functor of points as
\begin{equation}
S\mapsto
\left\{V=ker(\pi^*_{X^{(1)}}((F\times\id_S)_*(L))\twoheadrightarrow
G)| \Deg(V)=0, \rk(V)=r\right\},
\end{equation}
where $G$ is a coherent sheaf on $X^{(1)}$. By Grothendieck, this
functor is representable and proper.

Next I prove a simple Lemma which will be useful in subsequent constructions. This version of the lemma is taken from \citep*{wakabayashi13} as it is better than my original formulation (in a earlier version of this manuscript).

\begin{lemma}
The scheme $\Q^{r,0}_L$ carries a natural action of $\ker(\sV)$
given by $V\in\Q^{r,0}_L\mapsto V\tensor \gamma$ for any $\gamma\in
\sV$.
\end{lemma}
\begin{proof}
The lemma is evident from the definition of the action and the fact that
if $V\hookrightarrow F_*(L)$ then $V\tensor \gamma\hookrightarrow
F_*(L)\tensor\gamma=F_*(L\tensor\gamma^p)=F_*(L)$. Observe that
$\det(V\tensor \gamma)=\det(V)\tensor\gamma^r$.
\end{proof}

Let $\Q^{r,0}_{L,\gamma}$ be the closed subscheme of $\Q^{r,0}_L$
defined by the condition
\begin{equation}
V\in\Q^{r,0}_{L,\gamma}\Leftrightarrow V\in\Q^{r,0}_L, \text{\ and\
}\det(V)=\gamma^r.
\end{equation}
Then clearly $\Q^{r,0}_{L,\gamma}$ is closed in $\Q^{r,0}_L$ and
tensoring with  $\gamma'\in \ker(\sV)$ induces an isomorphism
$\Q^{r,0}_{L,\gamma}\to \Q^{r,0}_{L,\gamma\tensor\gamma'}$ for any
$\gamma,\gamma'\in \ker(\sV)$.

 Then, as $\Deg(\ker(\sV))=p^g$, one sees that
\begin{equation}\label{eq1}
\Deg(\Q^{r,0}_L)=p^g\Deg(\Q^{r,0}_{L,\gamma})
\end{equation}
 for any $\gamma\in
\ker(\sV)$.

\section{The dormant locus of $\SL(r)$-opers}
Let me digress in this section and describe the dormant locus of $\SL(r)$-opers (defined analogously). In other words
one considers vector bundles rather than their projectivizations. So I  indicate how to do this in this section. One may consider the dormant locus $\tDrgz$ in the stack of $\SL(r)$ opers. This is defined as follows: for any $k$-scheme $S$, one has a groupoid $\tDrgz(S)$ given by triples $(V,\nabla,V_\mydot)$ where $V$ is a vector bundle of rank $r$ on $X\times S$; $\nabla$ is an $\O_S$-linear connection on $V$ which is dormant and $V_\mydot$ is the oper flag on $V$, plus  an isomorphism of $\O_S$-linear connections $(\det(V),\det(\nabla))\isom(\O_{X\times S},\nabla=d)$. Morphisms over $S$ are isomorphisms of vector bundles $V$ which commute with the connection, the trivialization of the determinant and preserving the filtrations.  Tensoring with a line bundle of order $r$ gives a natural action of the stack of line bundles of order $r$ on $X$ (recall that this is the stack, which for any $k$-scheme $S$, associates the groupoid consisting of a line bundle $\eta$ on $X\times S$ and an isomorphism $\eta^r\simeq \O_{X\times S}$) on $\tDrgz$. I will now explicate action. Let $V\in\tDrgz(S)$ and let $\eta$ be a line bundle of order $r$ on $X$. Then $\eta$ carries a unique connection $\nabla_\eta$ such that $(\eta^r,\nabla_{\eta^r})=(\O_{X\times S},d)$. It is easy to see that the $p$-curvature of this connection on $\eta$ is zero (see \citep[Proposition 2.1.3]{joshi09}).

Equip $V\tensor\eta$ with the tensor product of the connection from $V$ and the connection of $\eta$. Let us note that the choice of Cartier connection on $\eta$ fixes for us a canonical isomorphism on the spaces of connections on $V$ and $V\tensor\eta$ (the two spaces are $H^0(X,\Omega^1\tensor\End(V))$ and $H^0(X,\Omega^1\tensor\End(V\tensor\eta))$ (respectively). Now equip $V\tensor\eta$ with the filtration induced from $V$. It is easy to see that this provides the structure of an oper on $V\tensor\eta$ and $(\det(V\tensor\eta),\nabla)=(\O_{X\times S},d)$. Thus one has for any pair $V\in\tDrgz$ and  $\eta$ a line bundle of order $r$ on $X$, a canonical oper $V\tensor \eta$ in $\tDrgz$  as claimed.

Further it is evident from our description that this action of line bundles of order $r$ is also transitive. This action is also compatible with the action of line bundles of order $r$  on the quotient line bundle  $V\to \gr^r(V)$ (such line bundles $\gr^r(V)$  also form a torsor under line bundles of order $r$). 

The structure of $\tDrgz$ is studied in detail in \citep{bd-hitchin,joshi09}. 
Now fix a line bundle $\theta\in{\rm Pic}(X)$ such that 
$$\theta^{\tensor 2}\isom\Omega^1_X$$
so $\theta$ is a theta characteristic on $X$. This provides a choice of $L$: let $L=\theta^{-(r-1)}$. Then one has an isomorphism $$L^{\tensor r}\tensor(\Omega^1_X)^{r(r-1)/2}\isom\O_X.$$ For a $k$-scheme $S$, let $p_X:X\times S\to X$ be the projection to $X$.

By the proof of \citep[Proposition~3.2.3]{joshi09}, the choice of $\theta$-characteristic provides an isomorphism of $\tDrgz$ with the stack $\sDrgz\times {\mathcal T}_r(X)$, where $\mathcal{T}_r(X)$ is the stack of line bundles of order $r$ on $X$. 

With the choice of $L$  and a line bundle $\eta$ on $X$ of order $r$, I define a substack of $\tDrgz$:
\begin{equation}
\tDrgz{}_{,L,\eta}(S)=\left\{V\in\tDrgz(S)| V\twoheadrightarrow
\gr^r(V)=p^*_X(L)\tensor\eta\right\}.
\end{equation}
This is clearly a  substack of $\tDrgz$. Further recall from \citep{bd-hitchin,joshi09} that an $SL(r)$-oper $(V,\nabla,V_\mydot)$ has ${\rm Aut}((V,\nabla,V_\mydot))=\mu_r$ and that any such automorphism is central in ${\rm Aut}(V)$. Thus it follows that objects of $\tDrgz{}_{,L,\eta}$ have no automorphisms.

The canonical arrow $\tDrgz\to\sDrgz$ given by projectivization of the $\SL(r)$-oper data, makes $\tDrgz$ into torsor (over $\sDrgz$) for the line bundles of order $r$ on $X$. Choice of $\theta$ gives us a section of $\tDrgz\to\sDrgz$  and induces an isomorphism  of $\tDrgz{}_{,L,0}$ and $\sDrgz$ which depends on the choice of $L$ (hence on the choice of $\theta$). 

The stack $\tDrgz{}_{,L,0}$ is thus  identified with the stack of dormant $\pgl(r)$-opers on $X$ (see \citep{joshi09}).  The stack of $\pgl(r)$-opers is a scheme and $\tDrgz$ is a closed subscheme of the scheme of $\pgl(r)$-opers. Since the latter is representable by an (affine) scheme it follows that $\tDrgz{}_{,L,0}$ is also represented by an (affine) scheme.

I will use this identification of $\tDrgz{}_{L,0}$ with $\sDrgz$ but suppress $\theta$ from the notation for convenience. 
In particular one sees that
\begin{equation}\label{eq2}
\Deg(\tDrgz)=r^{2g}\Deg(\tDrgz{}_{,L,0})=r^{2g}\Deg(\sDrgz) ,
\end{equation}
where $0\in J_X[r]$ is the identity element.

\section{A canonical isomorphism}
\begin{theorem}\label{eq3}
If $p>C(r,g)$ and $g\geq 2$ then the natural morphism
\begin{eqnarray*}
  \tDrgz{}_{,L,0} &\to & \Q^{r,0}_{L,0} \\
  V &\mapsto& V^\nabla\hookrightarrow F_*(L)\twoheadrightarrow F_*(L)/V^\nabla
\end{eqnarray*}
 is an isomorphism of  schemes.
\end{theorem}
\begin{proof}
Since $\tDrgz{}_{,L,0}$ is represented by a scheme it suffices to prove this at the level of functor of points.	Let $S$ be any $k$-scheme. I describe a natural bijection of sets $\tDrgz{}_{,L,0}(S)$ and $\Q^{r,0}_{L,0}(S)$.  Let $V\in\tDrgz{}_{,L,\eta}(S)$. This
means one has an oper $(V,\nabla,V_\mydot,\alpha)$ on $X\times S$.
Let me note that $F\times\id_S$ is the relative Frobenius of $X\times S$. Consider the subsheaf $V^\nabla\subset V$ of flat sections of $\nabla$.  The composing the inclusion $V^\nabla\into V$ with the surjection  $V\to V/V_1=L$ one gets a morphism $V^\nabla\to L$ of sheaves (but not of $\O_{X\times S}$-modules). For local sections $f$ of $\O_{X\times S}$ and $v$ of $V^\nabla$, one has $\nabla(f^pv)=df^p\tensor v+f^p\nabla(v)=0$. Thus if one considers $L$ as an $\O_{X\times S}$-module through the Frobenius $F\times\id_S$, the morphism of sheaves $V^\nabla\to V/V_1=L$ can be naturally viewed as a morphism of $\O_{X\times S}$-modules through the $F\times\id_S$, that is a morphism of locally free sheaves of $\O_{X\times S}$-modules $V^\nabla\to F_*(L)$.  As $\nabla$ is dormant, one has an isomorphism of vector bundles $(F\times\id_S)^*(V^\nabla)\isom V$ on $X\times S$ and hence $V^\nabla$ is a locally free of rank $r=\rk(V)$. By \citep*[Theorem~4.1.1(1)]{joshi09}, which applies by our assumption on $p$, one has $V^\nabla\hookrightarrow (F\times \id_S)_*(\pi_X^*(L))$ and this gives an object of the quot scheme $\Q^{r,0}_{L,0}(S)$. The construction of $V^\nabla\into F_*(L)\to F_*(L)/V^\nabla$ is clearly functorial in $S$. This gives us a functor $\tDrgz{}_{,L,0}(S)\to \Q^{r,0}_{L,0}(S)$. On the other hand
if $V'\in\Q^{r,0}_{L,0}(S)$ is an object of the quot-scheme, then this means one has a quotient
\be
0\to V'\to (F\times \id_S)_*(\pi_X^*(L))\to G\to 0.\ee
One notes that by \citep[Theorem~4.1.1(2)]{joshi09} every such $V'$ is semi-stable. Then adjunction gives a morphism of bundles $V=(F\times\id_S)^*(V')\to L$ and equipping $V$ with the Cartier connection $\nabla$ (which is dormant) and the canonical filtration \'a la \citep*[Section 5.3]{joshi06} one gets the oper-structure by Theorem 5.4.1, Prop. 3.4.2 and Remark 5.4.3 of \citep*{joshi09}.
This gives a dormant oper $(V=F\times\id_S)^*(V'),\nabla^{can}, (F\times\id_S)^*(V')_\mydot)\in\tDrgz{}_{,L,0}(S)$ on $X\times S$ whose construction is clearly functorial in $S$, and provides for each $S$ an inverse $\Q^{r,0}_{L,0}(S)\to\tDrgz{}_{,L,0}(S)$.  This completes the proof.
\end{proof}

\begin{corollary}
Let $X$ be a smooth, projective curve of genus $g\geq 2$ over an
algebraically closed field of characteristic $p>C(r,q)$. Let $L$ be
a line bundle fixed earlier. Then
\be
\Deg({\sDrgz(X)}_{L,0})=\frac{1}{p^g}\Deg(\Q^{r,0}_{L})
\ee
and
\be
\Deg({\tDrgz(X)}_{L,0})=\frac{r^{2g}}{p^g}\Deg(\Q^{r,0}_{L}).
\ee
\end{corollary}

\begin{proof}
The proof is immediate from \eqref{eq3}, \eqref{eq1} and
\eqref{eq2}.
\end{proof}

\section{Computing the schematic degree}
The preceding discussion reduces the problem of computing the
schematic degree of $\sD^r_{g,0}(X)$ to computing the degree of the
quot-scheme $\Q^{r,0}$. This is a special case of a more general
sort of problem which is very well-understood over complex numbers
under the assumption that the ambient bundles (whose) quot-schemes
are being studied are very general and the formula for the degree is
a special case of the \textit{Vafa-Intriligator Formula}
(see \citep*{intriligator91}, \citep*{holla04}). This formula
can be used, using results of \citep*{joshi09} and the results of the preceding sections,
to find a conjectural
formula for the degree of the dormant
operatic locus using the reductions made in the preceding section.
I show that my conjecture is true for $g=2,r=2$ (where the degree
of the dormant locus has been computed by different methods and
different authors).

Before I begin let us paraphrase what has been achieved in the
preceding sections. The problem of computing the degree has been
reduced to calculating the degree of the Quot-scheme $\Q^{r,0}$ of
quotients of $F_*(L)$ of rank $p-r$ and degree equal to
$\Deg(F_*(L))$.

This is a special case of the following problem, studied by \citep*{popa03},
\citep*{lange03}, \citep*{holla04} and several others (see references to papers
by these authors for a longer list).

Let $E$ be a vector bundle of rank $n$ and degree $d$. Let $1\leq
r\leq n$ be an integer. Let 
\be
e_{max}(E,r)=\max_{F}(\Deg(F))
\ee
where the maximum is taken over all subbundles $F\subset E$ of
rank $r$. Let 
\be\label{def:sr} s_r(E)=d\cdot r-n\cdot e_{max}(E,r).\ee 
Then by a well-known result of \citep*{mukai85}, one has
\be\label{eq:mukai-bound}
s_{r}(E)\leq r(n-r)g
\ee
and one can give a better estimate by a result of \citep*{hirschowitz88}. Under
the assumption that $E$ is very general and very-stable of degree
$d$ and rank $n$ one has
\begin{equation}\label{eq:hirschowitz}
s_r(E)=r(n-r)(g-1)+\varepsilon,
\end{equation}
where $\varepsilon$ is the unique integer satisfying the following
two conditions
\begin{enumerate}
  \item $0\leq\varepsilon<n$,
  \item $s_r(E)\equiv rd \bmod{n}$.
\end{enumerate}
The number $\varepsilon$, under these assumptions, is the dimension
of every irreducible component of the quot-scheme
$\Q^{r,e_{max}(E,r)}(E)$ of quotients of $E$ of rank $n-r$ and
degree $d-e_{max}(E,r)$. By maximality of $E$, every such quotient
is locally free and so the associated kernel is always a subbundle.
In particular if $s_r(E)=r(n-r)(g-1)$ then the quot-scheme is zero
dimensional and is known to be smooth (this is due to \citep*{popa03}).

Assume now that $s_r(E)=r(n-r)(g-1)$. Let
\be N(n,d,r,g)=\Deg(\Q^{r,e_{max}(E,r)}(E)).\ee Then one gets
\begin{eqnarray*}
  s_r(E) &=& r(n-r)(g-1) \\
  e_{max}(E,r) &=& dr-r(n-r)(g-1).
\end{eqnarray*}

In this setting a formula for the degree of this quot-scheme was
described by \citep*{holla04} in terms of Gromov-Witten theory; see \citep*{lange03} and its references for the rank two case. Recall that  
\begin{theorem}[{\citep[Theorem 4.2]{holla04}}]
Let $k=\C$, let $g\geq 2$, $E$ be a general stable
bundle of rank and degree as above.  Write $d=ar-b$ with $0\leq
b<r$. Then 
\begin{equation}\label{holla}
\Deg(\Q^{r,0}(E))=\frac{(-1)^{(r-1)(br-(g-1)r^2)/n}n^{r(g-1)}}{r!}\sum_{\zeta_1,\ldots,\zeta_r}\frac{\left(\prod_{i=1}^r\zeta_i\right)^{b-g+1}}{\prod_{i\neq
j}(\zeta_i-\zeta_j)^{g-1}},
\end{equation}
where $\zeta_i\in\C,\zeta_i^n=1$, for $1\leq i\leq r$ and the sum is over tuples
$(\zeta_1,\ldots,\zeta_r)$ with $\zeta_i\neq \zeta_j$.
\end{theorem}

In my setup I  want to take $E=F_*(L)$ with \be\Deg(L)=(1-r)(g-1).\ee
Then using
\begin{equation*}
\mu(F_*(L))=\frac{\Deg(L)}{p}+\frac{(p-1)(g-1)}{p},
\end{equation*}
one sees that $d=\Deg(E)$ is
\begin{equation*}
d=(p-r)(g-1),
\end{equation*}
and $0\leq \varepsilon< p$ and
\begin{equation*}
s_r(E)\equiv r(p-r)(g-1)\bmod{p}.
\end{equation*}
Thus one gets $r(p-r)(g-1)+\varepsilon\equiv r(p-r)(g-1)\mod{p}$, as
$0 \leq\varepsilon<p$ this means $\varepsilon=0$ (which is as it
should be because in our case one of the main results of \citep*{joshi09}
says that this quot-scheme is finite) and a similar calculation
shows that
\begin{eqnarray}
  e_{max}(E,r) &=& dr-r(p-r)(g-1) \\
   &=& r(p-r)(g-1)-r(p-r)(g-1)\\
   &=&0.
\end{eqnarray}
By \citep*{joshi09} the quot-scheme $\Q^{r,0}_L(E)$ is non-empty and
every point of this scheme provides  such a subbundle of degree
zero.

\section{The degree conjecture}
One can now write down the conjectural formula for the degree
of the dormant locus.
\begin{conj}\label{dormant-degree-formula}
Let $k$ be an algebraically closed field of characteristic
$p>C(r,g)$ and $X/k$ be a smooth projective curve over $k$ of genus
$g\geq 2$. Let $L$ be a line bundle such that
$L^{r}\tensor\left(\Omega^1_X\right)^{\frac{r(r-1)}{2}}=\O_X$. Then
the degree
\begin{eqnarray*}
\Deg(\sD^r_{g,0}(X))&=&\frac{1}{p^g}N(p,(p-r)(g-1),r,g)\\
&=&\frac{1}{p^g}\Deg(\Q^{r,0}_L)\\
&=&\frac{1}{p^g}\frac{p^{r(g-1)}}{r!}\sum_{\zeta_1,\ldots,\zeta_r}\frac{\left(\prod_{i=1}^r\zeta_i\right)^{(r-1)(g-1)}}{\prod_{i\neq
j}(\zeta_i-\zeta_j)^{g-1}}
\end{eqnarray*}
where $\zeta_i\in\C,\zeta_i^p=1$, for $1\leq i\leq r$ and the sum is over tuples
$(\zeta_1,\ldots,\zeta_r)$ with $\zeta_i\neq \zeta_j$.
\end{conj}
Note that I am conjecturing that the formula \citep*{holla04} and
the Vafa-Intriligator formula \citep*{intriligator91}, continues to
hold in characteristic $p$ and even holds for the quot-scheme of the specific
bundle $F_*(L)$ (which is not general).

To see that the exponents are what I claim they are, one notes that
here $n=p$ and on writing $d=(p-r)(g-1)=pa-b$ with $0\leq b<p$ gives
us \be d=p(g-1)-r(g-1)\ee as $p>C(r,g)$ with $b=r(g-1)$, $a=(g-1)$. The
power of $(-1)$ in Holla's formula cancels out with these choice of
parameters.

\section{The formula for genus two, rank two}
The degree of the quot-scheme in this special case ($g=2,r=2$ and $X$ is ordinary) was calculated by \citep*{lange08} by
an explicit method (which displays the degree in terms  Chern classes
of a suitable bundles) and also by Mochizuki in \citep*{mochizuki-foundations}, also see \citep*{osserman07} ($X$ general) by different methods. In \citep*{holla04} the number of maximal subbundles of a generic bundle was computed (for $g=2,r=2$) and this note began with my observation that these two numbers agree. Let me recall these known formulae here.

By \citep*{mochizuki-foundations,lange08,osserman07} one has
\begin{equation}
\Deg(\sD_2(X))=
\frac{(p^3-p)}{24}.
\end{equation}
On the other hand Holla's formula \citep*{holla04} gives  that
\be N(n,d,2,2)=\frac{n^3(n^2-1)}{24},\ee
which gives with $n=p$ and $d=p-2$, and so conjectural the degree of
$\Deg(\Q^{r,0}_L)$ is $\Deg(\Q^{r,0}_L)=N(p,p-2,2,2)$. This gives
with $g=2$ that
\be\Deg(\sD_2(X))=
\frac{1}{p^g} \frac{p^3(p^2-1)}{24},\ee
which is true  by the result of \citep*{mochizuki96}, \citep*{lange08} and \citep*{osserman07}.

\section{Additional Comments}\label{addcomments}
I learned from Wakabayashi's preprint that in \citep*{liu06} it was shown that for $r=2$, $\deg(\sDrgz)$ is a polynomial in $p$, of degree $3g-3$, with rational coefficients. This is based on a different approach (see \citep{wakabayashi13} for more comments on \citep*{liu06}). In \citep{wakabayashi13} one finds many explicit formulae for $\deg(\sDrgz)$ for small $g$.

The following remarkable fact is proved (for $r=2$) in \citep*{wakabayashi13}--the formula in Conjecture~\ref{dormant-degree-formula} is equivalent to
\be \Deg(\sDrgz)=2^{-g}\dim(H^0(\sU_X(2,\O_X),\theta^{p-2})\ee
where $\sU_X(2,\O_X)$ is the moduli space of semistable bundles of rank r=2 with trivial determinant and $\theta$ is the the theta-line bundle on this moduli space. The space on the right hand side of the above is the space of conformal blocks and its dimension is thus the Verlinde formula! (see \citep*{beauville94} for the terminology).

 As I was preparing to speak about this topic and Wakabayashi's work at a meeting in Nice in June 2013  (hosted by Christian Pauly), it seemed reasonable to me that this should hold for all ranks, and so I suggested at the meeting that Wakabayashi's observed coincidence should hold for all ranks. More precisely
\be\label{eq:verline-dormant-formula}
\dim(\sDrgz)=r^{-g}\dim H^0(\sU_X(r,\O_X),\theta^{p-r}).
\ee
Christian Pauly verified immediately after my lecture that the right hand side of  Conjecture~\ref{dormant-degree-formula}
is always the dimension of space of conformal blocks (up to the factor).  Thus one has a coincidence of two numbers whose genesis is rather complicated but one has no natural explanation for this numerical coincidence. At any rate one may view the above equality as a subtler form of Conjecture~\ref{dormant-degree-formula}. The connection between Verlinde numbers and the Vafa-Intrilligator formula is quite well-established (see \citep[Proposition~1]{marian07}).

This discussion begs the question: what is the most natural explanation for the coincidence of these two numbers in \eqref{eq:verline-dormant-formula}? What hidden relation exists between dormant opers and spaces of conformal blocks? Further let me note that the spaces $H^0(\sU(r)_X,\theta^m)$ of conformal blocks  satisfy \emph{Fusion rules} (which allow one to compute their dimensions combinatorially). Thus it seems reasonable to speculate that same sort of Fusion rules for computing $\sDrgz$ still lie hidden. One may, of course, infer these rules from the known rules for the right hand side and the observed coincidence, but this approach is not very insightful.

\bibliographystyle{plainnat}
\bibliography{dormant}
\end{document}